\documentclass[english, 12pt]{amsart}

\usepackage{tikz-cd}
\usepackage{aliascnt}
\usepackage{amsfonts}
\usepackage{mathrsfs}
\usepackage{amsthm}
\usepackage{amsmath,amssymb}
\usepackage[all,poly,knot]{xy}
\usepackage{amsfonts}
\usepackage{color}
\usepackage{hyperref}
\usepackage{comment}
\usepackage{csquotes}
\usepackage{amscd}
\usepackage{mathtools}
\usepackage{dsfont}

\usepackage[charter]{mathdesign}

\DeclareFontFamily{U}{dutchcal}{\skewchar\font=45 }
\DeclareFontShape{U}{dutchcal}{m}{n}{<-> s*[1.0] dutchcal-r}{}
\DeclareFontShape{U}{dutchcal}{b}{n}{<-> s*[1.0] dutchcal-b}{}
\DeclareMathAlphabet{\mathlcal}{U}{dutchcal}{m}{n}
\SetMathAlphabet{\mathlcal}{bold}{U}{dutchcal}{b}{n}

\usepackage{contour}
\usepackage{ulem}
\usepackage{amsmath,calligra,mathrsfs}
\DeclareMathOperator{\sheafhom}{\mathscr{H}\text{\kern -3pt {\calligra\large om}}\,}
\DeclareMathOperator{\sheafend}{\mathscr{E}\text{\kern -3pt {\calligra\large nd}}\,}

\theoremstyle{plain}
\newtheorem{thmx}{Theorem}

\newtheorem{conjx}[thmx]{Conjecture}

\theoremstyle{definition}
\newtheorem{definition}[subsubsection]{Definition}

\newtheorem{definition-proposition}[subsubsection]{Definition-Proposition}

\newtheorem{theorem}[subsubsection]{Theorem}
\newtheorem{example}[subsubsection]{Example}
\newtheorem{remark}[subsubsection]{Remark}

\theoremstyle{theorem}
\newtheorem{proposition}[subsubsection]{Proposition}
\newtheorem*{proposition*}{Proposition}
\newtheorem{lemma}[subsubsection]{Lemma}

\newtheorem{conjecture}[subsubsection]{Conjecture}
\newtheorem*{quest}{Question}

\newtheorem*{theorem*}{Theorem}

\makeatletter
\newcommand{\namelabel}[1]{%
  \phantomsection
  \renewcommand{\@currentlabel}{#1}
  \label{#1}
}
\makeatother

\newcommand{\thistheoremname}{}
\newtheorem*{genericthm*}{\thistheoremname}
\newenvironment{namedthm*}[1]{\renewcommand{\thistheoremname}{#1}%
	\begin{genericthm*}}
	{\end{genericthm*}}

\theoremstyle{remark}

\numberwithin{equation}{section}

\def\At{\mathrm{At(P)}}

\usepackage[top=1.5in, bottom=1.5in, left=1.5in, right=1.5in]{geometry}

 \usepackage[hyperpageref]{backref}
 
\usepackage{xcolor}
\hypersetup{
	colorlinks,
	linkcolor={red!50!black},
	citecolor={blue!50!black},
	urlcolor={blue!80!black}
}

\contourlength{0.8pt}

\newcommand{\p}{\partial}
\newcommand{\om}{\omega}
\newcommand{\cc}{\frac{1}{2}}
\newcommand{\bg}{g^{\star_{h_0}}}
\newcommand{\Mdol}{\mathcal M_{\mathrm{Dol}}(\mathcal X/S)}
\newcommand{\Ch}{D_{h_0}}
\newcommand{\Ehol}{\mathcal A^{0,0}(\mathrm{End}E\otimes\Omega^1_{X_0})}
\newcommand{\Eahol}{\mathcal A^{0,1}(\mathrm{End}E)}
\newcommand{\Hy}{\mathbb H^1(X_0,(\mathrm{End}E,\mathrm{ad}(\theta)))}
\newcommand{\Dol}{\mathrm{Dol}}
\newcommand{\End}{\mathrm{End}}
\newcommand{\SL}{\mathrm{SL}(n,\mathbb C)}
\newcommand{\tad}{\mathrm{ad}(\theta)}

\begin{document}
\title[Kodaira-Spencer map on Hitchin-Simpson correspondence]{Kodaira-Spencer Map on the Hitchin-Simpson Correspondence}

\author[Tianzhi Hu]{Tianzhi Hu}
\address{ School of Mathematics and Statistics, Wuhan University, Luojiashan, Wuchang, Wuhan, Hubei, 430072, P.R. China}
\email{hutianzhi@whu.edu.cn}


\author{Ruiran Sun}
 \address{School of Mathematical Sciences, Xiamen University, Xiamen 361005, China}
\email{ruiransun@xmu.edu.cn}

\author[Kang Zuo]{Kang Zuo}
\address{ School of Mathematics and Statistics, Wuhan University, Luojiashan, Wuchang, Wuhan, Hubei, 430072, P.R. China; Institut f\"ur Mathematik, Universit\"at Mainz, Mainz, Germany, 55099}
\email{zuok@uni-mainz.de}
\begin{abstract}
We define the isomonodromic deformation of a Higgs bundle over a compact Riemann surface via the Hitchin-Simpson correspondence and the isomonodromic deformation of a local system. This deformation defines a real analytic section of the relative Dolbeault moduli space, yielding a real analytic foliation on this moduli. This foliation generalizes the Betti foliation defined by the Betti map in the study of abelian schemes. We provide a precise form for the holomorphic and anti-holomorphic derivatives of the isomonodromic deformation of a Higgs bundle. Subsequently, we extend the classical non-abelian Kodaira-Spencer map using the anti-holomorphic derivative. Additionally, we prove that if the isomonodromic deformation of a graded Higgs bundle is not holomorphic, then the isomonodromically deformed Higgs field is non-nilpotent.
\end{abstract}


\maketitle

\setcounter{tocdepth}{1}
\tableofcontents

\section{Introduction}
\subsection{Hitchin-Simpson correspondence}
Let $X$ be a compact K\"ahler manifold. The celebrated Hitchin-Simpson correspondence, primarily developed by Corlette \cite{Corl}, Donaldson \cite{Don}, Hitchin \cite{Hit}, and Simpson \cite{Simp92}, establishes a real analytic homeomorphism between the moduli space of semisimple $\mathbb C$-local systems and the moduli space of polystable Higgs bundles with vanishing Chern classes. This correspondence is denoted by
$$\Psi:\mathcal M_{\mathrm{B}}(X)\stackrel{\sim}{\longrightarrow} \mathcal M_{\mathrm{Dol}}(X),$$
where we consider $\SL$ local systems and Higgs bundles unless otherwise stated. 

This correspondence is \textbf{real analytic} because its construction involves solving a highly non-linear second-order elliptic PDE for both directions. 

Furthermore, the classical Riemann-Hilbert correspondence over $X$ provides a biholomorphic map $\mathcal M_{\mathrm{B}}(X)\simeq\mathcal M_{\mathrm{DR}}(X)$, where the latter is the moduli space of semisimple $\SL$ flat bundles over $X$.

\subsection{Non-abelian Hodge theory for a projective family}\label{nabHT}
For a smooth projective family $f:\mathcal X\to S$ over a quasi-projective base, with a central fiber $X_0$ over $0\in S$, Griffiths \cite{Gri69, Gri70} developed the theory of variation of Hodge structures, which is crucial in deformation theory and Hodge theory (see \cite{CMP}). 

Inspired by Deligne, Simpson \cite{Simp95} introduced a non-abelian analogue of classical variation of Hodge structures. Deligne's twistor space is defined as 
$$\mathcal P:\mathcal M_{\mathrm{Del}}(\mathcal X/S)\to S\times\mathbb P^1,$$ 
whose fiber at $S\times \{1\}$ is $\mathcal M_{\mathrm{DR}}(\mathcal X/S)$ and at $S\times \{0\}$ is $\Mdol$. Both relative moduli spaces are generally not smooth, even if $\mathcal X/S$ is a smooth family. In this paper, all tangent bundles of these relative moduli spaces are defined only at their smooth points.

In \cite{Simp95}, the \textbf{non-abelian Gauss-Manin connection} $\nabla_{GM}$ is defined via the isomonodromic deformation on $\mathcal M_{\mathrm{Del}}(\mathcal X/S)|_{S\times\mathbb G_m}$, and the \textbf{non-abelian Hodge filtration} is the $\mathbb G_m$-action on $\mathcal M_{\mathrm{Del}}(\mathcal X/S)|_{S\times\mathbb G_m}$. Consequently, $\Mdol$ can be interpreted as the '\textbf{non-abelian Hodge bundle}' of $\mathcal M_{\mathrm{DR}}(\mathcal X/S)$ due to the Rees module construction. 

Let $\mathcal M^{\mathrm{gr}}_{\mathrm{Dol}}(\mathcal X/S)\subset\Mdol$ denote the moduli of graded Higgs bundles. In \cite[Lemma 4.1]{Simp10}, Simpson proved that for any $(E,D)\in \mathcal M_{\mathrm{Del}}(X_0)|_{\mathbb A^1}$, 
$$\lim\limits_{t\in\mathbb G_m,\ t\to 0}(E,t\cdot D)\text{ exists and is contained in } \mathcal M^{\mathrm{gr}}_{\mathrm{Dol}}(X_0).$$ 

Let $p_0:\mathcal M^{\mathrm{gr}}_{\mathrm{Dol}}(\mathcal X/S)\to S$ be the projection and $\rho_{KS}$ be the Kodaira-Spencer map of $\mathcal X/S$. 

As proposed in \cite{Simp95}, and by taking the residue at $S\times \{0\}$ of the non-abelian Gauss-Manin connection (as in \cite{CT} and \cite{FS}), we obtain the following \textbf{non-abelian Kodaira-Spencer map} (or non-abelian Higgs field): 
\begin{equation}\label{Theta_ks}\begin{aligned}\Theta_{KS}:p_0^\ast TS&\to T({\Mdol/S})\\ \{(E,\bar\p,\theta),\frac{\p}{\p t}\}&\mapsto \theta_\ast\circ\rho_{KS}(\frac{\p}{\p t}),
\end{aligned}\end{equation}
where $(E,\bar\p,\theta)\in\mathcal M_{\mathrm{Dol}}^{\mathrm{gr}}(X_s)$, $\frac{\p}{\p t}\in T_sS,$ and 
\begin{align*}\theta_\ast:H^1(X_s,T_{X_s})\to\mathbb H^1(X_s,(\mathrm{End}E,\mathrm{ad}(\theta)))=T_{(E,\bar\p,\theta)}\mathcal M_{\mathrm{Dol}}(X_s)
\end{align*}
is the induced map from the morphism of two complexes (each column is a complex)
\[
\begin{tikzcd}
0 \arrow[r, ]                                                   & \mathrm{End} E \otimes \Omega_{X_s}^1 \\
T_{X_s} \arrow[r, "\theta"] \arrow[u, ] & \mathrm{End} E \arrow[u, "{\tad}"]              
\end{tikzcd}
\]
We remark that $\theta_\ast$ is defined for any Higgs bundle, not necessarily graded.

Let $p:\Mdol\to S$. We offer a different approach to derive and naturally extend the non-abelian Kodaira-Spencer map $\Theta_{KS}$. 

By performing the \textbf{isomonodromic deformation of a Higgs bundle} (explained in section \ref{1-3}), we obtain a real analytic section $\sigma:S\to\Mdol$. Its anti-holomorphic derivative at $\sigma(s)$ along $\p/\p t\in T_s S$ is given by
\begin{equation}\label{Phi_ks}\begin{aligned}\Phi^{0,1}_{KS}:p^\ast TS&\to T^{1,0}({\Mdol/S})\xrightarrow{\text{conjugate}} T^{0,1}({\Mdol/S})\\ \{(E,\bar\p,\theta),\frac{\p}{\p t}\}&\mapsto \theta_\ast\circ\rho_{KS}(\frac{\p}{\p t}) \xrightarrow{\text{conjugate}} \overline{\theta_\ast\circ\rho_{KS}(\frac{\p}{\p t})},
\end{aligned}\end{equation}

Note that $\Theta_{KS}$ in \eqref{Theta_ks} is defined only for a graded Higgs bundle, while $\Phi^{0,1}_{KS}$ is defined for any Higgs bundle. Therefore, $\Phi_{KS}^{0,1}$ \textbf{extends} $\Theta_{KS}$ to any Higgs bundle in $\Mdol$.

\subsection{Isomonodromic deformation and the first main result}\label{1-3}
As our focus is on the local properties of isomonodromic deformation, we assume $S$ is a germ of some variety. An isomonodromic section $\tau:S\to\mathcal M_{\mathrm{DR}}(\mathcal X/S)\cong\mathcal M_{\mathrm{B}}(\mathcal X/S)$ represents an isomonodromic deformation for a given flat vector bundle (or a local system). 

The Hitchin-Simpson correspondence $\Psi$ provides the \textbf{isomonodromic deformation section} of Higgs bundle, defined by:
$$\sigma:=\Psi\circ\tau:S\to \Mdol.$$

We refer to this $\sigma$ as an \textbf{isomonodromic deformation} of a Higgs bundle $\sigma(0)=(E,\bar\p,\theta)\in\mathcal M_{\mathrm{Dol}}(X_0)$. This section $\sigma$ is a real analytic section, rather than holomorphic, because $\Psi$ is real analytic. All isomonodromic sections of $\Mdol$ define a real analytic foliation on $\Mdol$, called the \textbf{isomonodromic foliation} on $\Mdol$.

The Betti map of an abelian scheme, introduced in \cite[Section 2.1]{CMZ}, is a valuable tool in studying Diophantine problems, including the distribution of torsion values in an abelian scheme \cite{CMZ, ACZ} and the geometric Bogomolov conjecture \cite{CGHX}. We observe that (in Section \ref{Bettimap}):
\begin{itemize}
\item The Betti map of an abelian scheme $\mathcal A\to S$ is equivalent to a real analytic map 
$$\mathcal B:\mathcal A\to \mathcal M_{\mathrm B}(\mathcal A/S,\mathrm U(1))\cong S\times \text{a fixed real torus }.$$
defined in \eqref{Bet} by the unitary monodromy representation.
\item When considering the relative Jacobian $\mathrm{Jac}(\mathcal X/S)$ of a family of compact Riemann surfaces $\mathcal X/S$, the Betti map is a special case of the Hitchin-Simpson correspondence:
\[\begin{tikzcd}\mathrm{Jac}(\mathcal X/S)\arrow[r,"\mathcal B"]\arrow[d,equals]&\mathcal M_{\mathrm B}(\mathrm{Jac}(\mathcal X/S)/S,\mathrm U(1))\arrow[d,equals]\\
\mathcal M_{\mathrm {Dol}}(\mathcal X/S,\mathrm U(1))\arrow[r]&\mathcal M_{\mathrm B}(\mathcal X/S,\mathrm U(1))\end{tikzcd}\]
The Betti foliation on $\mathrm{Jac}(\mathcal X/S)$, defined by the level sets of the Betti map, is precisely the isomonodromic foliation on $\mathcal M_{\mathrm {Dol}}(\mathcal X/S,\mathrm U(1))$.
\end{itemize}

By \cite[Proposition 2.1]{CMZ} and \cite[Section 2.1]{CGHX}, any level set of a Betti map is a holomorphic section of the abelian scheme. However, an isomonodromic deformation $\sigma:S\to\Mdol$ is not always holomorphic.
\begin{quest} When is an isomonodromic deformation $\sigma:S\to\Mdol$ holomorphic?\end{quest}
\begin{flushleft}\textbf{Setup:} We focus on the smooth projective family of compact Riemann surfaces $\mathcal X/S$. Even if, the main ingredient in our paper, i.e. using the deformation of the harmonic metric to study the deformation of the Higgs bundle, should work for any higher dimensional family.\end{flushleft}

After reviewing the deformation theory of a Higgs bundle in $\Mdol$ using the theory of the Atiyah bundle (developed in \cite{Ati}) and studying the first-order deformation of the harmonic metric, we arrive at the following explicit representation for the tangent map of $\sigma$, which answers the question regarding holomorphicity.  
\begin{thmx}[=Theorem~\ref{main1}]\label{thm-A}
(i) The \textbf{holomorphic derivative} $\Pi^{1,0}\sigma_\ast(\frac{\p}{\p t})$ of the isomonodromic section $\sigma:S\to\Mdol$ is given by
\begin{align*}\Phi^{1,0}_{KS}:T_0S&\to T^{1,0}_{\sigma(0)}\Mdol\\
\frac{\p}{\p t}&\mapsto [(-\cc\Ch^{1,0}g,-\cc[\theta^{\star_{h_0}},g],\eta)],\end{align*}
where this triple is a deformation class defined in Proposition \ref{holtan}. $[\eta]:=\rho_{KS}(\p/\p t)$ is the Kodaira-Spencer class. $g$ is a smooth endomorphism of $E$ satisfying the PDE \eqref{equ_g} in Proposition \ref{thetadbar} such that $g$ is uniquely determined by $\frac{\p}{\p t}$ and the initial Higgs bundle $\sigma(0)\in\Mdol$;\\
(ii) The \textbf{anti-holomorphic derivative} $\Pi^{0,1}\sigma_\ast(\frac{\p}{\p t})$ for the isomonodromic section $\sigma:S\to\Mdol$ is 
\begin{align*}\Phi^{0,1}_{KS}:T_0S\stackrel{\rho_{KS}}{\longrightarrow}H^1(T_{X_0})\stackrel{\theta_\ast}{\longrightarrow}T_{\sigma(0)}^{1,0}\mathcal M_{\mathrm{Dol}}(X_0)\stackrel{\text{conjugate}}{\longrightarrow}T_{\sigma(0)}^{0,1}\mathcal M_{\mathrm{Dol}}(X_0)\hookrightarrow T_{\sigma(0)}^{0,1}\Mdol.
\end{align*}
\end{thmx}
We refer to the above derivatives $\Phi^{1,0}_{KS}$ and $\Phi^{0,1}_{KS}$ as the \textbf{non-abelian Kodaira-Spencer maps} because they arise directly from the isomonodromic deformation and the Hitchin-Simpson correspondence.

The PDE that $g$ satisfies is a linear non-homogeneous second-order elliptic PDE, similar to the classical $\partial\bar\partial$-equation in complex geometry. The complex conjugation of $\Phi^{0,1}_{KS}$ is indeed a sheaf morphism $ p^\ast TS\to T(\Mdol/S)$. But the holomorphic part $\Phi^{1,0}_{KS}$ seems to be \textbf{transcendental} even in the rank one case, i.e. $\mathcal M_{\mathrm{Dol}}(\mathcal X/S,\mathbb C^\ast)$.

As a direct corollary of Theorem \ref{thm-A}, the isomonodromic deformation of any unitary Higgs bundle is holomorphic, as $\Phi^{0,1}\equiv 0$.

One may wonder what happens if the isomonodromic deformation of any Higgs bundle is holomorphic. We propose the following conjecture: 
\begin{conjx}[=Conjecture~\ref{nfconj}]\label{conj-B} Let $\Mdol^{0}$ be an irreducible component of $\Mdol$. Then the following statements are equivalent:\\
(a) Any isomonodromic deformation section $\sigma: S\to \Mdol^{0}$ is a holomorphic section;\\
(b) $\Mdol^{0}$ is locally trivial over $S$.
\end{conjx}
For the rank 1 case, i.e., $\mathcal M_{\mathrm{Dol}}(\mathcal X/S,\mathbb C^\ast)$, Conjecture \ref{conj-B} reduces to the classical Torelli theorem (see Remark \ref{torelli}).
\subsection{Graded Higgs bundles and the second main result}
A graded Higgs bundle over $X$ is equivalent to the Hodge bundle of a polarized $\mathbb C$-variation of Hodge structures over $X$ (see \cite{CMP}). Simpson \cite{Simp92, Simp10} also proved that graded Higgs bundles are precisely those Higgs bundles invariant under the $\mathbb C^\ast$ action: 
$$(E,\bar\p,\theta)\stackrel{\lambda\in\mathbb C^\ast}{\longrightarrow}(E,\bar\p,\lambda\theta).$$

We consider the isomonodromic deformation $\sigma:S\to\Mdol$ of a graded Higgs bundle $\sigma(0) \in\mathcal M_{\mathrm{Dol}}(X)$.

\begin{quest} Is the isomonodromically deformed Higgs bundle still graded?\end{quest}

This question probes the commutativity between the $\mathbb C^\ast$ action and the isomonodromic deformation of a Higgs bundle. The answer is generally no, as \cite[Theorem 1.4]{KYZ} proved that $\Phi_{KS}^{0,1}$ is precisely the obstruction for $\sigma(t)$ to be a graded Higgs bundle for $t$ near $0$. 

We can prove the following stronger result by applying Theorem \ref{thm-A} and a classical Bochner technique:
\begin{thmx}[=Theorem~\ref{nonnil}]\label{thm-C} Assume (a) $\sigma(0):=(E,\bar\p,\theta)$ is a graded Higgs bundle on $X$, and \\(b) the non-abelian Kodaira-Spencer map $\Phi^{0,1}_{KS}$ is non-zero at $\sigma(0)$ for some $\frac{\p}{\p t}\in T_0S$.\\
Then the isomonodromic deformation of $(E,\bar\p,\theta)$ along the direction $\frac{\p}{\p t}\in T_0S$ is always non-nilpotent.
\end{thmx}
In our proof of Theorem \ref{thm-C}, we show that the image of $\Phi^{0,1}_{KS}$ is contained in the horizontal direction (defined in \cite[section 1.3]{MT}) of the Hitchin fibration map. Hence, we think it is possible to relax condition (a) in Theorem \ref{thm-C} and we propose the following conjecture:
\begin{conjx}[=Conjecture~\ref{nil_nonnil}]\label{conj-D} Assume (a) $\sigma(0):=(E,\bar\p,\theta)$ is a nilpotent Higgs bundle on $X_0$ and \\(b) the non-abelian Kodaira-Spencer map $\Phi^{0,1}_{KS}$ is non-zero at $\sigma(0)$ for some $\frac{\p}{\p t}\in T_0S$.\\
Then the isomonodromic deformation of $(E,\bar\p,\theta)$ is always non-nilpotent along the direction $\frac{\p}{\p t}\in T_0S$.
\end{conjx}
Let $\mathcal T_g$ be the Teichm\"uller space of genus $g\geq 2$ compact Riemann surfaces and let $\mathcal X/\mathcal T_g$ be its universal family. Since $\mathcal T_g$ is contractible, any isomonodromic section $\sigma:\mathcal T_g\to \mathcal M_{\mathrm{Dol}}(\mathcal X/\mathcal T_g)$ is well-defined. 

We may ask the {\bf distribution} of nilpotent Higgs bundles when doing the isomonodromic deformation:
\begin{conjx}[=Conjecture~\ref{disconj}]\label{conj-E} We say the isomonodromic section $\sigma:\mathcal T_g\to \mathcal M_{\mathrm{Dol}}(\mathcal X/\mathcal T_g)$ has nilpotent loci $\mathcal N\subset \mathcal T_g$ if $\mathcal N$ is a union of all maximal irreducible real analytic subvarieties $S_i\subset \mathcal T_g,\ i\in I$, such that $\sigma|_{S_i}$ is always nilpotent. We expect that $\{S_i:i\in I\}$ is a {\bf discrete} set of subvarieties. \\
(The word 'discrete' means that at any point $s\in \mathcal T_g$, there is a Euclidean open neighborhood $U$ of $s$, such that $\{i\in I:\ U\cap S_i\ne\emptyset\}$ is a finite set.)
\end{conjx}
\section{Preliminaries}
Now we review some basic theory and notations. For a background on Hodge theory and Higgs bundles, one may refer to \cite{CMP}. For a survey paper on harmonic maps and Higgs bundles, one may refer to \cite{Li}.
\subsection{Local system and vector bundle}
Let \( X_0 \) be a compact Riemann surface of genus $g\geq 2$ with a K\"ahler form $\omega_0$. For any holomorphic (smooth) vector bundle $E$ over $X_0$, let $\mathcal A^{p,q}(E):=C^\infty(E\otimes \Omega^{p,q}_{X_0})$ be the space of smooth $(p,q)$ sections of $E$ on $X_0$.

\hfill

\noindent
\textbf{Flat bundle}: a tuple \( (E, D) \), where \( E \) is a holomorphic vector bundle on $X_0$ with a flat holomorphic connection $D:\mathcal O(E)\to \mathcal O(E\otimes \Omega_{X_0}^1)$. After tensoring $C^{\infty}(X_0)$, $(E,D)$ is a smooth flat vector bundle with a flat connection $D:\mathcal A^0(E)\to\mathcal A^1(E)$. In this paper, a flat bundle refers to a smooth flat bundle unless otherwise stated.

\hfill

\noindent
\textbf{$\mathbb C$-local system}: a locally constant sheaf over $X_0$ with stalk \( V \), where $V$ is a complex vector space. For a local system \( \mathcal V\) over \( X_0 \), then $ \mathcal V \otimes_{\mathbb{C}} \mathscr C^{\infty}_{X_0}$ is a locally free sheaf of \( \mathscr C^{\infty}_{X_0} \)-modules, which is a flat smooth vector bundle. Similarly, $\mathcal V \otimes_{\mathbb{C}} \mathcal{O}_{X_0}$ is a locally free sheaf of \( \mathcal{O}_{X_0} \)-modules, which is a flat holomorphic vector bundle.  

\hfill

\noindent
\textbf{Chern connection}: given a holomorphic vector bundle \( (E, \bar\p) \) with a Hermitian metric \( h \), there exists a unique unitary connection \( D_{h} \) such that \( D_{h}^{0,1} = \bar\p \), which is called the Chern connection. Let $F(D_h):=D_h\circ D_h$ be the Chern curvature, which is a (1,1) form of $\mathrm{End}E$.

\hfill

\noindent
\textbf{Hermitian metric on \( \mathcal A^1(\mathrm{End} E) \)}: let $(E,h)$ be a Hermitian vector bundle. For any \( \varphi, \psi \in \mathcal A^0(\mathrm{End} E) \), $\alpha,\beta\in \mathcal A^1(X_0)$  
\[
\langle \varphi \otimes \alpha, \psi \otimes \beta \rangle := \mathrm{tr}(\varphi \psi^{\star_{h}}) \cdot \Lambda_{\omega_0}(\alpha \wedge *\beta)
\]  
where $*$ is the Hodge star operator with respect to $\omega_0$, $\star_h$ is the Hodge star operator with respect to $h$, and $\Lambda_{\omega_0}$ is the contraction by $\omega_0$.

\subsection{Harmonic metric}
Given a flat vector bundle \( (E, D) \), and any Hermitian metric $h$ on it,  
we have the following unique decomposition
\[
D = D_h + \psi_h
\] 
where \( D_h \) is unitary, and \( \psi_h \) is self-adjoint, decomposed by type (1,0) and (0,1):  
\begin{align*}
    D  = D_h + \psi_h  = D_h^{1,0} + D_h^{0,1} + \theta + \theta^{\star_h}.
\end{align*}

Define the energy functional 
\[
E(h) := \int_{X_0} \langle \psi_h, \psi_h \rangle \, \omega_0.
\]  

If \( h \) is a minimal point of \( E(h) \), we call it a \textbf{harmonic metric}. Such $h$ satisfies the following equation
$$D_h^{\star_h}\psi_h=0.$$

If a harmonic metric $h$ on $(E,D)$ exists, then  
\( (E, D_h^{0,1}, \theta) \) is a Higgs bundle.

\begin{theorem*}[Donaldson, Corlette]
    A semisimple flat bundle $(E,D)$ admits a harmonic metric. 
\end{theorem*}

\subsection{Higgs bundle}
A Higgs bundle over $X_0$ is a triple \( (E, \bar{\partial}, \theta) \), where \( E \) is a holomorphic vector bundle on $X_0$ with a holomorphic structure \( \bar{\partial} \), and \( \theta \in \mathcal A^{1,0}(\mathrm{End} E) \) satisfies \( \bar{\partial}\theta = 0 \) and \( \theta \wedge \theta = 0 \). (When $X_0$ is a compact Riemann surface, the condition $\theta\wedge\theta=0$ is automatically satisfied.)

We say $h$ is a harmonic metric for a Higgs bundle \( (E, \bar{\partial}, \theta) \) if $h$ satisfies the Hitchin-Yang-Mills equation \begin{align}\label{HYM}F(D_h)+[\theta, \theta^{\star_h}]=0.\end{align}

If a harmonic metric $h$ on \( (E, \bar{\partial}, \theta) \) exists, then $(E,D_h+\theta+\theta^{\star_h})$ is a flat bundle.
\begin{theorem*}[Hitchin, Simpson]
    A polystable Higgs bundle with vanishing Chern classes $(E,\bar\p, \theta)$ admits a harmonic metric. 
\end{theorem*}

\subsection{Moduli spaces and correspondences} One may refer to \cite{Simp92} and \cite{Simp94I}.
\hfill

\begin{flushleft}$\mathcal{M}_\mathrm{B}(X_0)$: moduli space of semisimple representations of $\pi_1 (X_0)$ into $\SL$.\end{flushleft}

\hfill

\noindent
$\mathcal{M}_{\mathrm{dR}}(X_0)$: moduli space of semisimple $\SL$ flat vector bundles over $X_0$.

\hfill

\noindent
$\mathcal{M}_{\mathrm{Dol}}(X_0)$: moduli space of polystable $\SL$ Higgs bundles with vanishing Chern classes on $X_0$.

\hfill

\noindent
\textbf{Riemann-Hilbert correspondence}: we have the biholomorphic map
\begin{align*}
    \mathcal{M}_{\mathrm{B}}(X_0) & \to \mathcal{M}_{\mathrm{DR}}(X_0),\\ 
    [\rho:\pi_1(X_0)\to \SL] & \mapsto [\widetilde{X_0}\times_\rho \mathbb{C}^n],
\end{align*}
where $\widetilde{X_0}$ is the universal covering of $X_0$.

\hfill

\noindent
\textbf{Hitchin-Simpson correspondence}: we have the real analytic map 
\begin{align*}
   \Psi: \mathcal{M}_{\mathrm{DR}}(X_0) & \to \mathcal{M}_{\mathrm{Dol}}(X_0),\\ 
    [(E, D)] & \mapsto [(E, D_h^{0, 1}, \theta)].
\end{align*}

\subsection{Holomorphic tangent space of $\mathcal{M}_{\mathrm{Dol}}(X_0)$}\label{htdol}
Let $(E,\bar\p,\theta)\in \mathcal M_{\mathrm{Dol}}(X_0)$ be a smooth point. The holomorphic tangent space \( T_{(E, \bar{\partial}, \theta)} \mathcal{M}_{\mathrm{Dol}}(X_0) \) can be characterized by the hypercohomology $\Hy$. We shall use the \textbf{Dolbeault resolution} of the complex $\mathrm{End}E\stackrel{\tad}{\longrightarrow}\mathrm{End}E\otimes\Omega^1_{X_0}$ to compute this hypercohomology. The Dolbeault resolution is 
\[
\begin{tikzcd}
C^{0,1}:=\Eahol \arrow[r, "{\tad}"]                                                   & C^{1,1}:={\mathcal A^{0,1}(\mathrm{End} E \otimes \Omega_{X_0}^1)} \\
C^{0,0}:=\mathcal A^{0,0}(\mathrm{End} E) \arrow[r, "{\tad}"] \arrow[u, "\bar\partial"] & C^{1,0}:=\Ehol \arrow[u, "\bar\partial"]              
\end{tikzcd}
\]
Then we have the following truncated complex
\[C^{0,0} \xrightarrow{d^0} C^{1,0} \oplus C^{0,1} \xrightarrow{d^1}  C^{1,1},\]
where 
\begin{gather*}
    d^0(g) = ([g,\theta], \bar{\partial} g) \quad \text{for } \quad g \in C^{0,0},\\
    d^1(\varphi,\psi) = \bar{\partial} \varphi + [\psi,\theta] \quad \text{for } \quad (\varphi,\psi) \in C^{1,0} \oplus C^{0,1}.
\end{gather*}
Then we have the holomorphic tangent space
$$T_{(E, \bar{\partial}, \theta)} \mathcal{M}_{\mathrm{Dol}}(X_0)=\Hy=\frac{\mathrm{Ker\ }d^1}{\mathrm{Im\ }d^0}.$$

\subsection{Relative moduli spaces}
Let \( f: \mathcal{X} \to S \) be a smooth projective family of compact Riemann surfaces over \( S \) with center fiber $X_0$ over $0\in S$. Simpson constructed three relative moduli spaces in \cite{Simp94II}. They are the relative Betti moduli, the relative de Rham moduli, and the relative Dolbeault moduli: 

$$\mathcal{M}_\mathrm{B}(\mathcal{X}/S),\ \mathcal{M}_{\mathrm{DR}}(\mathcal{X}/S),\ \mathcal{M}_{\mathrm{Dol}}(\mathcal{X}/S),$$
whose fibers at $s\in S$ are
$$\mathcal{M}_\mathrm{B}(X_s),\ \mathcal{M}_{\mathrm{DR}}(X_s),\ \mathcal{M}_{\mathrm{Dol}}(X_s),\text{ respectively}.$$
By \cite[Proposition 7.8]{Simp94II}, we have the complex analytic homeomorphism \begin{align}\label{RH}\mathcal{M}_\mathrm{B}(\mathcal{X}/S)\cong \mathcal{M}_{\mathrm{DR}}(\mathcal{X}/S).\end{align}
And by the Hitchin-Simpson correspondence, we have the real analytic homeomorphism
$$ \Psi:\mathcal{M}_{\mathrm{DR}}(\mathcal{X}/S)\stackrel{\sim}{\longrightarrow}\Mdol.$$
\subsection{Isomonodromic deformation}\label{Isomd}
By Ehresmann's theorem, locally \( \mathcal{X}/S \) has a \( C^\infty \)-trivialization \( \mathcal{X}|_U \cong U \times X_0 \), where \( U \) is a small neighborhood of $0\in S$. Then for each \( s \in U \), this trivialization gives a diffeomorphism \( F_s : X_s \to X_0 \). To avoid the monodromy of the isomonodromic deformation, we shall assume $S$ is a germ of a positive dimension variety near $0$.

For \( \rho \in \mathcal{M}_\mathrm B(X_0) \), the pull back \( F_s^* \rho \in \mathcal{M}_\mathrm B(X_s) \), which gives a section of \( \mathcal{M}_B(\mathcal{X}/S) \to S \), called the \textbf{isomonodromic deformation} of \( \rho \) over \( S \), denoted by the section
$$\tau:S\to \mathcal{M}_B(\mathcal{X}/S).$$

The Riemann-Hilbert correspondence gives a biholomorphic map \( \mathcal{M}_\mathrm B(X_s) \cong\mathcal{M}_{\mathrm{DR}}(X_s) \). Under this, we get a section of \( \mathcal{M}_{\mathrm{DR}}(\mathcal{X}/S) \to S \), called the \textbf{isomonodromic deformation of a flat vector bundle} $(\mathcal V,D)\in\mathcal{M}_{\mathrm{DR}}(X_0) $, still denoted by $\tau:S\to \mathcal{M}_{\mathrm{DR}}(\mathcal X/S).$
\begin{definition}(a) Recall the Hitchin-Simpson correspondence $$ \Psi:\mathcal{M}_{\mathrm{DR}}(X_s) \to \mathcal{M}_{\mathrm{Dol}}(X_s).$$ Under this real analytic correspondence, we obtain a real analytic section $$\sigma:=\Psi\circ \tau:S\to \mathcal{M}_{\mathrm{Dol}}(\mathcal{X}/S),$$ called the \textbf{isomonodromic deformation of a Higgs bundle} $\sigma(0)\in\mathcal M_{\mathrm{Dol}}(X_0).$\\
(b) The \textbf{isomonodromic foliation} of $\Mdol$ is a real analytic foliation generated by all its isomonodromic sections.
\end{definition}
We remark that the isomonodromic deformation can be defined for any smooth projective family of relative dimension $\geq1$.
\subsection{Non-abelian Gauss-Manin connection and non-abelian Kodaira-Spencer map}\label{naGMnaKS}
Let $\pi:\mathcal{M}_{\mathrm{DR}}(\mathcal X/S)\to S$ be the relative de Rham moduli over $S$. We define the non-abelian Gauss-Manin connection on $\mathcal{M}_{\mathrm{DR}}(\mathcal X/S)$ as a horizontal lifting (see \cite{CT} and \cite{FS})
\begin{align*}\nabla_{\mathrm{GM}}: \pi^\ast TS&\to T\mathcal{M}_{\mathrm{DR}}(\mathcal X/S)\\
\{(E,D),\frac{\p}{\p t}\}&\mapsto\tau_\ast(\frac{\p}{\p t}),
\end{align*}
where $(E,D)\in\mathcal M_{\mathrm{DR}}(X_s)$, $\frac{\p}{\p t}\in T_sS$, and $\tau:S\to\mathcal{M}_{\mathrm{DR}}(\mathcal X/S)$ is the isomonodromic deformation of $(E,D)\in\mathcal M_{\mathrm{DR}}(X_s)$ near $s\in S.$ For simplicity, we write the above map as $\nabla_{\mathrm{GM}}:\frac{\p}{\p t}\mapsto\nabla_{\mathrm{GM},\frac{\p}{\p t}}$. We can define its curvature as 
$$F(\nabla_{\mathrm{GM}})(\frac{\p}{\p t},\frac{\p}{\p s}):=[\nabla_{\mathrm{GM},\frac{\p}{\p t}},\nabla_{\mathrm{GM},\frac{\p}{\p s}}]-\nabla_{\mathrm{GM},[\frac{\p}{\p t},\frac{\p}{\p s}]}.$$

By \eqref{RH}, $\nabla_{\mathrm{GM}}$ is a holomorphic connection. Additionally, $\nabla_{\mathrm{GM}}$ has zero curvature, meaning isomonodromic sections indeed define a holomorphic foliation on $\mathcal{M}_{\mathrm{DR}}(\mathcal X/S)$. Thus, $\mathcal{M}_{\mathrm{DR}}(\mathcal X/S)\to S$ is locally trivial over $S$. By \cite[Section 7]{Simp95}, this $\nabla_{GM}$ extends to $\mathcal M_{\mathrm{Del}}(\mathcal X/S)|_{S\times\mathbb G_m}$. Recall the non-abelian Kodaira-Spencer map $\Theta_{KS}$ defined in \eqref{Theta_ks} of section \ref{nabHT}. 

Now we provide a formula for $\theta_\ast$ using the Dolbeault representative (see \cite[Proposition 2.1]{Zuo} and \cite[Theorem 1.2]{FS}).
\begin{lemma}\label{thetaast}Let $\eta\in \mathcal A^{0,1}(T_{X_s})$ be a representative of $[\eta]\in H^1(X_s,T_{X_s})$. Then $$\theta_\ast([\eta])=[(0,\eta(\theta))]\in\mathbb H^1(X_s,(\mathrm{End}E,\mathrm{ad}(\theta))),$$
where $\eta(\theta)\in \mathcal A^{0,1}(\mathrm{End}E)$ is a contraction by vector fields.
\end{lemma}

\subsection{Betti map and isomonodromic deformation}\label{Bettimap}
Assume our base $S$ is a germ of a variety. Let $\mathcal A\to S$ be an abelian scheme of relative dimension $g$. We can find a basis of relative holomorphic 1-forms of $\mathcal A\to S$, denoted by $\{\omega_1(s),\cdots,\omega_g(s)\}_{s\in S}$. Since all fibers $\mathcal A_s$ have the same topological type, we can take $\gamma_1,\cdots,\gamma_{2g}$ as the common generators set of their fundamental groups.
\begin{definition}[\cite{CMZ} and \cite{CGHX}] (a) For any $\xi\in \mathcal A_s$, its coordinates $\{\int_0^\xi\omega_i(s)\}_{i=1}^g$ are well defined modulo $\bigoplus\limits_{j=1}^{2g}\mathbb Z\{
\int_{\gamma_j}\omega_i(s)\}_{i=1}^g $, and clearly we have $(b_1(\xi),\cdots,b_{2g}(\xi))\in \mathbb R^{2g}/\mathbb Z^{2g}$, such that $\int_0^\xi=\sum_{j=1}^{2g}b_j(\xi)\int_{\gamma_j}.$ We define the \textbf{Betti map} by 
\begin{align*}b:\mathcal A&\to\mathbb R^{2g}/\mathbb Z^{2g}\\
\xi&\mapsto(b_1(\xi),\cdots,b_{2g}(\xi)),\end{align*}
which is a real analytic map. \\
(b) We say a section $\alpha:S\to \mathcal A$ is a level set of $b$ if $b(\alpha(s))$ is a constant map. Through any $\xi\in\mathcal A_s$, there is a unique level set of $b$ passing through $\xi$. Thus all level sets of $b$ define a real analytic foliation on $\mathcal A$, called the \textbf{Betti foliation}.
\end{definition}
Since $\mathbb R/\mathbb Z\cong \mathrm{U}(1)$, for each $\xi\in \mathcal A_s$, $b(\xi)$ determines an element $\rho_\xi$ in $\mathrm{Hom}(\pi_1(\mathcal A_s),\mathrm{U}(1))=\mathcal M_{\mathrm B}(\mathcal A_s,\mathrm U(1))$ by mapping 
$$\gamma_j\to b_j(\xi),\ j=1,2,\cdots,2g.$$
So we have the following real analytic homeomorphism 
\begin{equation}\label{Bet}\begin{aligned}\mathcal B:\mathcal A&\to\mathcal M_{\mathrm B}(\mathcal A/S,\mathrm{U}(1))\\
\xi&\mapsto\rho_\xi.\end{aligned}\end{equation}
By fixing the generators of $\pi_1(\mathcal A_s),\ \text{for any } s\in S$, $\mathcal M_{\mathrm B}(\mathcal A/S,\mathrm{U}(1))$ is real analytically homeomorphic to $S\times \mathbb R^{2g}/\mathbb Z^{2g}$. Then $b$ and $\mathcal B$ are equivalent via this trivialization.\\
\textbf{Observation:} By the above definition, if $\alpha:S\to\mathcal A$ is a level set of $b$, then $\rho_{\alpha(\cdot)}:S\to\mathcal  M_{\mathrm B}(\mathcal A/S,\mathrm{U}(1))$ is an isomonodromic section. Conversely, any isomonodromic section of $\mathcal M_{\mathrm B}(\mathcal A/S,\mathrm{U}(1))$ defines a level set of the Betti map $b$.

Therefore, the isomonodromic deformation in the relative Betti moduli is a generalization of the level set of the Betti map. 

If $\mathcal X/S$ is a smooth family of compact Riemann surfaces of genus $g\geq2$, then we consider the abelian scheme $\mathrm{Jac}(\mathcal X/S)\to S$. In this case, the map 
$$\mathcal B:\mathrm{Jac}(\mathcal X/S)\to \mathcal M_{\mathrm B}(\mathrm{Jac}(\mathcal X/S)/S,\mathrm{U}(1))=\mathcal M_{\mathrm B}(\mathcal X/S,\mathrm{U}(1))$$ defined in \eqref{Bet} is simply the Hitchin-Simpson correspondence restricted to 
$$\mathrm{Jac}(\mathcal X/S)=\mathcal M_{\mathrm {Dol}}(\mathcal X/S,\mathrm{U}(1))\subsetneq \Mdol.$$

Therefore, a level set of the Betti map $b$ on $\mathrm{Jac}(\mathcal X/S)$ is exactly an isomonodromic section of $\mathcal M_{\mathrm {Dol}}(\mathcal X/S,\mathrm{U}(1))$. The Betti foliation coincides with the isomonodromic foliation on $\mathcal M_{\mathrm {Dol}}(\mathcal X/S,\mathrm{U}(1))$.
\section{Deformation theory for a Higgs bundle}\label{DT}
Let $f:\mathcal X\to S$ be a smooth projective family of compact Riemann surfaces with center fiber $X_0:=f^{-1}(0)$, for $0\in S$. We consider a real analytic family of Higgs bundles over $\mathcal X$ such that for each $s\in S$, this family restricted to $X_s$ is a Higgs bundle on $X_s$.

We have the relative Dolbeault moduli space $p:\Mdol\to S$. We may view a real analytic (or holomorphic) family of Higgs bundles as a real analytic (or holomorphic) section $\sigma: S\to\Mdol$ such that $p\circ\sigma=\mathrm {id}_S$. In the following, we always assume $$\sigma(0)=(E,\bar\p,\theta)\in \Mdol|_0=\mathcal M_{\mathrm{Dol}}(X_0).$$

\begin{lemma}\label{realsec}
When $\sigma: S\to\Mdol$ is a real analytic section, we consider its holomorphic and anti-holomorphic derivative along $v\in T_0S\ (=T_0^{1,0}S)$, i.e.,
$$\sigma_\ast(v)=w^{1,0}\oplus w^{0,1}\in T^{1,0}_{\sigma(0)}\Mdol\oplus T^{0,1}_{\sigma(0)}\Mdol.$$
Then $$p_\ast(w^{1,0})=v\text{ and }p_\ast(w^{0,1})=0.$$ In other words, the anti-holomorphic derivative of $\sigma$ along $v$ factors through $T_{\sigma(0)}^{0,1}\mathcal M_{\mathrm{Dol}}(X_0)$. Clearly, when $w^{0,1}=0$, $\sigma$ is holomorphic along $v\in T_0S$.
\end{lemma}
\begin{proof} Since $p$ is holomorphic, $p_\ast$ maps $T^{1,0}_{\sigma(0)}\Mdol$ to $T^{1,0}_0S$ and $T^{0,1}_{\sigma(0)}\Mdol$ to $T^{0,1}_0S$. Thus $$(p\circ \sigma)_\ast(v)=p_\ast(w^{1,0})+p_\ast(w^{0,1})\in T^{1,0}_0S\oplus T^{0,1}_0 S.$$ Also, $(p\circ \sigma)_\ast(v)=(\mathrm {id}_S)_\ast(v)=v\in T^{1,0}S$. So $p_\ast(w^{0,1})=0$.
\end{proof}

Later, we shall provide explicit deformation classes for $w^{1,0}$ and $w^{0,1}$, respectively.
\subsection{Deformation theory for a compact Riemann surface}\label{Rie}
In this section, we briefly review some basic deformation theory. Let $f:\mathcal X\to S$ be a family as above with central fiber $X_0:=f^{-1}(0)$, for $0\in S$. Then we have the Kodaira-Spencer map
\begin{align}\label{KS}
\rho_{KS}: T_0S\to H^1(X_0,T_{X_0})
\end{align}

We use the Dolbeault cohomology to represent any class $[\eta]\in H^1(X_0,T_{X_0})$ as follows:
$$H^1(X_0,T_{X_0})=\frac{\mathcal A^{0,1}(T_{X_0})}{\bar\p \mathcal A^{0,0}(T_{X_0})}.$$

Let $\frac{\p}{\p t}\in T_0S$ such that $\rho_{KS}(\frac{\p}{\p t})=[\eta]$.
We may pick some $\eta\in \mathcal A^{0,1}(T_{X_0})$ to represent $[\eta]\in H^1(X_0,T_{X_0})$. 

Now we assume $X_t:=f^{-1}(t)$ is sufficiently close to the central fiber with the complex structure $\mu_t:=t\eta+O(t^2)\in \mathcal A^{0,1}(T_{X_0})$ on $X_t$. We view $X_t$ as the differential manifold $X_0$ endowed with the complex structure $\mu_t$. Then the identity map $$\mathrm{id}:X_0\to (X_0,\mu_t)$$
is a diffeomorphism.

We let $dz$ be a local $(1,0)$ frame on $X_0$ and $d\bar z$ its conjugate. Then $$\om_t:=dz-t\eta(dz)+O(t^2)$$ is a local $(1,0)$ frame for $X_t$, where $\eta(dz)\in\mathcal A^{0,1}(X_0)$ is the contraction.

\subsection{Deformation class of a holomorphic family of Higgs bundles}

In this section, we assume $\sigma: S\to\Mdol$ is a holomorphic section and $\sigma(0)=(E,\bar\p,\theta)\in \mathcal M_{\mathrm{Dol}}(X_0)$. 

We will use the theory of the Atiyah bundle to provide the first-order deformation class of a Higgs bundle in $\Mdol$. One may refer to \cite{Ati} for the definition of the Atiyah class. In \cite[Proposition 4.2.1]{CT}, the author used this theory to describe the deformation class of a triple $(X_0,P,\nabla)$, where $P$ is a $G$-principal bundle on $X_0$ with a holomorphic connection $\nabla$.

For the initial Higgs bundle $(E,\bar\p,\theta)\in \mathcal M_{\mathrm{Dol}}(X_0)$, let $P\stackrel{j}{\to}X_0$ be its frame bundle, which is an $\SL$-principal bundle. Then we have the following short exact sequence $$0\to T_{P/X}\to T_P\to j^*T_{X_0}\to 0.$$
$G:=\SL$ acts on both tangent spaces. After taking a quotient by $G$, we have 
\begin{align}\label{at} 0\to \mathrm{ad}P\to \At\to T_{X_0}\to 0.\end{align}
We remark that $\mathrm{ad}P=\mathrm{End}E$ and $\At\cong \mathrm{ad}P\oplus T_{X_0}$ as a smooth vector bundle. 
An Atiyah class is an extension class $$a(E)\in\mathrm{Ext}^1(T_{X_0},\mathrm{End}E)\cong H^1(X,\mathrm{End}E\otimes \Omega_{X_0}^1).$$ 

Let $h_0$ be the harmonic metric for $E$ and $\Ch$ be the Chern connection with $\Ch=\Ch^{1,0}+\bar\p$. In our case, by \cite[Proposition 4]{Ati}, we may take 
$$a(E):=-\theta\wedge\theta^{\star_{h_0}}-\theta^{\star_{h_0}}\wedge\theta,$$
which is equal to $F(\Ch)$ by the Hitchin-Yang-Mills equation \eqref{HYM}.
This $a(E)$ defines a complex structure on $\At$ by $$\bar\p_{\At}:=\begin{pmatrix}\bar\p_{\mathrm{End}E} & a(E)\\ 0 &\bar\p_{T_{X_0}}\end{pmatrix}.$$

Let $\theta\lrcorner (\cdot):\mathcal A^{p,q}(T_{X_0})\to \mathcal A^{p,q}(\mathrm{End}E)$ be the contraction map. The operator $\Ch^{1,0}(\theta\lrcorner)$ for any smooth (local) section $\tau$ of $T_{X_0}\otimes\mathcal A^{p,q}(X_0)$ is defined by first performing the contraction $\theta\lrcorner\tau$ and then taking $\Ch^{1,0}(\theta\lrcorner \tau).$

\begin{lemma}\label{2termcx} We have the following 2-term complex 
\begin{align}\label{2-tcx}
\At\xrightarrow{\tad\oplus \Ch^{1,0}(\theta\lrcorner)}\mathrm{End}E\otimes\Omega_{X_0}^1,
\end{align}
which gives the following extension of two deformation complexes
\[\begin{tikzcd}[column sep=2cm]\mathrm{End}E\arrow[r,"\tad"]\arrow[d]&\mathrm{End}E\otimes\Omega_{X_0}^1\arrow[d,equals]\\
\At \arrow[r,"\tad\oplus \Ch^{1,0}(\theta\lrcorner)"]\arrow[d]& \mathrm{End}E\otimes\Omega_{X_0}^1\arrow[d]\\T_{X_0}\arrow[r,"0"]&0
\end{tikzcd}\]
\end{lemma}
\begin{proof} As a smooth map, $\tad\oplus \Ch^{1,0}(\theta\lrcorner):\At\cong  \mathrm{End}E\oplus T_{X_0}\to \mathrm{End}E\otimes\Omega_{X_0}^1$ is well-defined. We verify it is indeed holomorphic, i.e., the following diagram commutes:
\begin{equation}\label{reso}\begin{tikzcd}[column sep=2cm]
\mathcal A^{0,1}(\At)\arrow[r,"\tad\oplus (-\Ch^{1,0}(\theta\lrcorner))"] & \mathcal A^{0,1}(\mathrm{End}E\otimes\Omega_{X_0}^1) \\ \mathcal A^{0,0}(\At)\arrow[r,"\tad\oplus \Ch^{1,0}(\theta\lrcorner)"]\arrow[u,"\bar\p_{\At}"]&\mathcal A^{0,0}(\mathrm{End}E\otimes\Omega_{X_0}^1)\arrow[u,"\bar\p"]
\end{tikzcd}\end{equation}
For any $(g,\tau)\in \mathcal A^{0,0}(\At)=\mathcal A^{0,0}(\mathrm{End}E)\oplus \mathcal A^{0,0}(T_{X_0}),$ we have 
\begin{align*}\{\tad\oplus (-\Ch^{1,0}(\theta\lrcorner))\}\circ\bar\p_{\At}(g,\tau)=&\tad(\bar\p g)-\Ch^{1,0}(\theta\lrcorner(\bar\p\tau))+\tad\circ a(E)\tau;\\
\bar\p\circ\{\tad\oplus \Ch^{1,0}(\theta\lrcorner)\}(g,\tau)=&\bar\p(\tad g)+\bar\p\Ch^{1,0}(\theta\lrcorner\tau).
\end{align*}
Note that $\tad\circ a(E)\equiv0$ and $$\Ch^{1,0}(\theta\lrcorner(\bar\p\tau))+\bar\p\Ch^{1,0}(\theta\lrcorner\tau)=[F(\Ch),\theta\lrcorner\tau]=0.$$
We have the desired holomorphic property. The rest is clear by definition.
\end{proof}
Now we present the main result of this section.
\begin{proposition}\label{holtan}Under the assumption $T_0S=H^1(X_0,T_{X_0})$, the hypercohomology group $\mathbb H^1(\At, \tad\oplus \Ch^{1,0}(\theta\lrcorner))$ of \eqref{2-tcx} gives the first order deformation class of $(E,\bar\p,\theta)\in\Mdol$, i.e., \begin{align}\label{Holtan}T_{(E,\bar\p,\theta)}\Mdol=\mathbb H^1(\At, \tad\oplus \Ch^{1,0}(\theta\lrcorner)).\end{align} We also have the exact sequence
\begin{align}\label{shortH}0\to \Hy\to \mathbb H^1(\At, \tad\oplus \Ch^{1,0}(\theta\lrcorner))\stackrel{j_\ast}{\rightarrow}H^1(X_0,T_{X_0})\to 0
.\end{align}
\end{proposition}
\begin{proof}Consider the truncated complex of \eqref{reso}, we have
\begin{align*}\mathcal A^{0,0}(\At)\xrightarrow{D^0}\mathcal A^{0,1}(\At)\oplus \mathcal A^{0,0}(\mathrm{End}E\otimes\Omega^1_{X_0})\xrightarrow{D^1}\mathcal A^{0,1}(\mathrm{End}E\otimes\Omega^1_{X_0}),\\
\text{where\ } D^0=(\bar\p_{\At},\tad\oplus \Ch^{1,0}(\theta\lrcorner));\ D^1:=\{\tad\oplus (-\Ch^{1,0}(\theta\lrcorner))\}\oplus\bar\p.
\end{align*}
This truncated complex \textbf{computes} the hypercohomology group in \eqref{Holtan}. So any class in $\mathbb H^1(\At, \tad\oplus \Ch^{1,0}(\theta\lrcorner))$ can be represented by a $D^1$-closed cocycle \begin{align*}(\varphi,\psi,\eta)\in &\Ehol\oplus\Eahol\oplus \mathcal A^{0,1}(T_{X_0})\\&=\mathcal A^{0,0}(\mathrm{End}E\otimes\Omega^1_{X_0})\oplus\mathcal A^{0,1}(\At),\end{align*}
denoted by $[(\varphi,\psi,\eta)]\in \mathbb H^1(\At, \tad\oplus \Ch^{1,0}(\theta\lrcorner)).$

We give an elementary proof of \eqref{Holtan}. 

For any $t\in S$, let $(E,\bar\p_t,\theta_t)$ be the Higgs bundle given by $\sigma(t)$ over $X_t$. As in the previous section, $\mathrm{id}:X_0\to X_t$ is a diffeomorphism, and we use it to pull back $(E,\bar\p_t,\theta_t)$ to a triple on $X_0$. One tries to compare the pull-back triple $(E,\bar\p_t,\theta_t)$ with $(E,\bar\p,\theta)$ on $X_0$.\\
\textbf{Claim:} There exists $(\varphi,\psi,\eta)\in \Ehol\oplus\Eahol\oplus \mathcal A^{0,1}(T_{X_0})$ such that
\begin{align*}X_t\text{ has } &t\eta+O(t^2)\in \mathcal A^{0,1}(T^{1,0}_{X_0})\text{ as its complex structure};\\
\bar\p_t=&\bar\p+t\eta\circ\Ch^{1,0}-\bar t\bar\eta\circ \bar\p+t\psi+O(|t|^2);\\
\theta_t=&\theta-t\eta(\theta)+t\varphi+O(|t|^2).
\end{align*}

Note that for any $h\in C^\infty(X_0)$, the $\bar\p_t$ operator for $X_t$ is given by
$$\bar\p_t h=\bar\p h+t\eta\circ\p h-\bar t\bar\eta\circ\bar\p h+O(|t|^2).$$

Thus for $\bar\p_t$ of $E$ on $X_t$, we have $\bar\p_t=\bar\p+t\eta\circ\Ch^{(1,0)}-\bar t\bar\eta\circ \bar\p+t\psi+O(|t|^2)$ for some $\psi$ being a smooth (0,1)-form of $\mathrm{End}E$. 
Now we assume $\theta_t=(A_0+tA_1+O(t^2))\omega_t$ where $A_0$ and $A_1$ are smooth sections of $\mathrm{End}E$ and $\om_t$ is a (1,0) local frame of $X_t$ defined in section \ref{Rie}. Then 
$$\theta_t=A_0dz+tA_1dz-t\eta(A_0dz)+O(t^2).$$
Thus we must have $A_0dz=\theta$, and this completes the proof of our claim.

Now we consider the compatible condition $\bar\p_t\theta_t=0$ for the cocycle $(\varphi,\psi,\eta)$ in our claim. For any smooth section $B\omega_t$ of $\mathrm{End}E\otimes \Omega^{1,0}_{X_t}$ and $a\in\mathcal A^{0,0}(E)$, we have
\begin{align*}&\big(\bar\p_t(B\omega_t)\big)(a)=(\bar\p_tB)(a)\omega_t+B(a)\bar\p_t\omega_t\\
=&\{\bar\p_t(B(a))-B(\bar\p_ta)\}\omega_t+(Ba)\bar\p_t\omega_t\\
=&(\bar\p B)\wedge (dz)(a)-t\Ch^{1,0}(\eta(Bdz))(a)+t[\psi,Bdz](a)+O(|t|^2).
\end{align*}Take $B\omega_t=\theta_t$ in the above equality and we get 
\begin{align}\label{dgla}
\bar\p\varphi+[\psi,\theta]=\Ch^{1,0}(\eta(\theta)),
\end{align}
where $\eta:\mathcal A^{p,q}(\mathrm{End}E)\to\mathcal A^{p-1,q+1}(\mathrm{End} E)$ is the contraction. This \eqref{dgla} is just the \textbf{closed condition} $D^1(\varphi,\psi,\eta)=0$ defined in the truncated complex of \eqref{reso}.

Besides, $(g,0)\in \mathcal A^{0,0}(\mathrm{End} E)\oplus\mathcal A^{0,0}(T_{X_0})=\mathcal A^{0,0}(\At)$ defines an exact class$$D^0(g,0)=(\tad(g),\bar\p g,0).$$
This exact term arises from a gauge transform $G:=\mathrm{id}_E+tg+O(t^2)$ of $(E,\bar\p,\theta)$, for any such $g$. (If we consider $(0,\tau)\in\mathcal A^{0,0}(\mathrm{End} E)\oplus\mathcal A^{0,0}(T_{X_0})$ with $\tau\ne0$, the induced exact term $D^0(0,\tau)$ comes from the diffeomorphism of $X_0$ homotopic to the identity map of $X_0$.) This completes the proof of \eqref{Holtan}.

By Lemma \ref{2termcx}, we have the induced long exact sequence of hypercohomology groups. Since $H^0(X_0,T_{X_0})=0$ and $\mathbb H^2(X_0,(\mathrm{End}E,\mathrm{ad}(\theta)))=0$, we have \eqref{shortH}.
\end{proof}

\begin{remark}
1. If $T_0S$ is not equal to $H^1(X_0,T_{X_0})$, we let $\mathcal L:=\rho_{KS}(T_0S)\subset H^1(X_0,T_{X_0})$. Then the deformation class of $(E,\bar\p,\theta)\in\Mdol$ is $$j_\ast^{-1}(\mathcal L)\subset\mathbb H^1(\At, \tad\oplus \Ch^{1,0}(\theta\lrcorner)),$$
where $j_\ast$ is defined in \eqref{shortH}.\\
2. Fix $\frac{\p}{\p t}\in T_0S$. Let 
\begin{align}\label{Aspace}
A_{\p/\p t}:=\{[(\varphi,\psi,\eta)]\in \mathbb H^1(\At, \tad\oplus \Ch^{1,0}(\theta\lrcorner)):\ [\eta]=\rho_{KS}(\frac{\p}{\p t})\}\end{align}
which is equal to 
$\{v\in T^{1,0}_{\sigma(0)}\Mdol:p_\ast(v)=\frac{\p}{\p t}\}.$ For simplicity, we write a class in $A_{\p/\p t}$ as $[(\varphi,\psi)]$ omitting $\eta$. For any holomorphic section $\sigma:S\to\Mdol$, we have $$\sigma_\ast(\p /\p t)=[(\varphi,\psi)]\in A_{\p /\p t}.$$
3. Since $H^2(X_0,\mathrm{End}E)=0$, we may take $\psi=0$ in \eqref{dgla} to get a 'trivial' deformation $(E,\bar\p_t)$ of the holomorphic bundle $(E,\bar\p)$. Then \eqref{dgla} is the equation
$$\bar\p\varphi=\Ch^{1,0}(\eta(\theta)),$$
which coincides with the deformation equation in \cite[Theorem 1.3]{LRY} solving a holomorphic $(1,0)$ section of $(\mathrm{End}E,\bar\p_t)$ for each $t$ with the initial section $\theta\in \Gamma(\mathrm{End}E\otimes\Omega_{X_0}^1)$.

\end{remark}
\subsection{Deformation class of a real analytic family of Higgs bundles}
In this section, we assume $\sigma: S\to\Mdol$ is a real analytic section. We focus on its anti-holomorphic deformation class.

First, we recall the anti-holomorphic involution of Deligne's twistor space $\mathcal{M}_\mathrm{Del}(X_0) \to \mathbb{P}^1$ in \cite{Simp95}. 
Define the involution map $\sigma': \mathcal{M}_\mathrm{Del}(X_0) \to \mathcal{M}_\mathrm{Del}(X_0)$ by 
\[
(E,\lambda\Ch^{1,0}+\bar\p+\theta+\lambda\theta^{\star_{h_0}},\lambda) \mapsto (\overline E^\vee,-\bar\lambda^{-1}\Ch^{1,0}+\bar\p+\theta-\bar\lambda^{-1}\theta^{\star_{h_0}},-\bar\lambda^{-1}).
\]
We restrict this involution to $\mathcal M_{\mathrm{Dol}}(X _0)\to\{0\}$, 
which gives an anti-holomorphic homeomorphism
\begin{align*}
    \sigma':\mathcal M_{\mathrm{Dol}}(X_0) & \to \mathcal M_{\mathrm{Dol}}(\overline {X_0})\\
    (E,\bar\p,\theta) & \mapsto (\overline E^\vee,\Ch^{1,0},\theta^{\star_{h_0}})
\end{align*}
where $\overline{E}^\vee$ is the vector bundle whose transition functions are inverse to the conjugate transpose of transition functions of $E$, and $\overline{X_0}$ is the Riemann surface diffeomorphic to $X_0$ with the conjugate complex structure.

This gives a $\mathbb C$-linear isomorphism between $T^{0,1}_{(E,\bar{\partial}, \theta)} \mathcal M_{\Dol}(X_0)$ and $T^{1,0}_{(\overline E^\vee,\Ch^{1,0},\theta^{\star_{h_0}})} \mathcal M_{\Dol}(\overline{X_0})$, i.e.
\[
\sigma_{*,(E,\bar{\partial}, \theta)} : \overline \Hy \stackrel{\simeq}{\rightarrow} \mathbb H^1(\overline{X_0}, (\End \overline E^\vee, \mathrm{ad}(\theta^{\star_{h_0}}))).
\]

We want to use the Dolbeault resolution to compute $\overline\Hy$. So we view $\overline E^\vee\cong E$ by the harmonic metric $h_0$, and then $\mathrm{End}(\overline E^\vee)\cong\mathrm{End}E$. In addition, we clearly have $\Omega_{\overline{X_0}}^{1,0}=\Omega_{X_0}^{0,1}$ and $\Omega_{\overline{X_0}}^{0,1}=\Omega_{X_0}^{1,0}$. Hence, just as in section \ref{htdol}, we also have the following Dolbeault resolution:
\[
\begin{tikzcd}
C^{1,0}:=\mathcal A^{1,0}(\End E  ) \arrow[r, "{\mathrm{ad}( \theta^{\star_{h_0}})}"]                                                   & C^{1,1}:={\mathcal A^{1,0}(\End E \otimes \Omega_{X_0}^{0,1})} \\
C^{0,0}:=\mathcal A^{0,0} (\End E) \arrow[r, "{\mathrm{ad}( \theta^{\star_{h_0}})}"] \arrow[u, "\Ch^{1,0}"] & C^{0,1}:=\mathcal A^{0,0} (\End E \otimes \Omega_{X_0}^{0,1}) \arrow[u, "\Ch^{1,0}"]              
\end{tikzcd}
\]
which gives the following truncated complex  
\[C^{0,0} \overset{d^{0c}}{\longrightarrow} C^{1,0} \oplus C^{0,1} \overset{d^{1c}}{\longrightarrow}  C^{1,1}\]  
where 
\begin{gather*}
    d^{0c}(g) = (\Ch^{1,0} g, [\theta^{\star_{h_0}},g]) \quad \text{for } g \in C^{0,0},\\
    d^{1c}(\varphi, \psi) =  \Ch^{1,0} \psi +[\varphi,\theta^{\star_{h_0}}] \quad \text{for } (\varphi, \psi) \in C^{0,1} \oplus C^{1,0}.
\end{gather*}
Hence
$$T^{0,1}_{(E, \overline{\partial}, \theta)} \mathcal{M}_{\mathrm{Dol}}(X_0)=\overline\Hy=\frac{\mathrm{Ker\ }d^{1c}}{\mathrm{Im\ }d^{0c}}.$$
And we have the following complex conjugate operator 
\begin{align*}
    \mathrm{bar} : \Hy & \rightarrow \overline\Hy\\
    [(\varphi,\psi)] & \mapsto [(\psi^{\star_{h_0}}, -\varphi^{\star_{h_0}})].
\end{align*}

We shall describe how a class $[(\varphi,\psi)]\in\overline\Hy$ deforms $(E,\bar\p,\theta)$ in an anti-holomorphic manner. For this, we need to define the harmonic 1-form.

\begin{definition}\label{harcl}
    Suppose $(\varphi,\psi) \in C^{1,0}\oplus C^{0,1}$, we say $(\varphi,\psi)$ is harmonic if $\bar\p\varphi + [g,\theta] =0$ and $ \Ch^{1,0}\psi + [\varphi,\theta^{\star_{h_0}}] = 0$.
\end{definition} 

Then Hitchin \cite{Hit} proved that 

\begin{proposition}
    For each class in $\Hy$ (or in $\overline{\Hy}$), there is a unique harmonic representative $(\varphi , \psi)$ of this class as defined in Definition \ref{harcl}.
\end{proposition}

Any harmonic class $[(\varphi,\psi)]\in\overline\Hy$ defines a (first-order) anti-holomorphic deformation of $(E,\bar\p, \theta)$ by
$$\bar\p_t=\bar\p+\bar t\psi+O(|t|^2);\ \theta_t=\theta+\bar t\varphi+O(|t|^2).$$

Combining this with Proposition \ref{holtan} and Lemma \ref{realsec}, we have
\begin{proposition}\label{aholtan}Let $\sigma:S\to\Mdol$ be a real analytic section and $\frac{\p}{\p t}\in T_0S$. Then there exists a class $[(\varphi_1,\psi_1)]\in A_{\p/\p t}$ (defined in \eqref{Aspace}) and a harmonic class $[(\varphi_2,\psi_2)]\in\overline\Hy$ such that $\sigma_\ast(\frac{\p}{\p t})$ has $[(\varphi_1,\psi_1)]\in A_{\p/\p t}$ as its (1,0) part and has $[(\varphi_2,\psi_2)]$ as its (0,1) part. 
\end{proposition}

\section{Isomonodromic deformation of a Higgs bundle}
In this section, we prove our main results using the deformation theory discussed earlier. The main technique is to relate every deformation class to the first-order deformation of the harmonic metric of a given Higgs bundle.
\subsection{Deformation class and Kodaira-Spencer map on Hitchin-Simpson correspondence}

In this section, let $\sigma: S\to\Mdol$ be the isomonodromic deformation of $(E,\bar\p,\theta)$ on $X_0$ defined in section \ref{Isomd}. By the Hitchin-Simpson correspondence, $\sigma$ is a real analytic section. We shall provide explicit formulas for its holomorphic and anti-holomorphic deformation class.

Let $(\mathcal V,D)$ be the flat rank $n$ complex vector bundle over $X_0$ corresponding to $(E,\bar\p,\theta)$. (If we forget the holomorphic structure, $\mathcal V=E$ over $X_0$.) Let $h_0$ be the harmonic metric of $(\mathcal V,D)$ over $X_0$. Then $(\mathcal V,D)$ over $X_t$ is the isomonodromy deformation of the initial flat bundle $(\mathcal V,D)$ over $X_0$. Now we let $h$ be the harmonic metric of $(\mathcal V,D)$ over $X_t$, which satisfies the following equation
\begin{align}\label{Hmetric}
    D_h^{\star_h}\psi_h=0,
\end{align}
where $D=D_h+\psi_h$ is the decomposition of $D$ into unitary and Hermitian parts, and $\star_h$ is the Hodge star operator with respect to $h$. 

\begin{lemma}\label{Taylor}
    The harmonic metric $h$ has the following Taylor expansion with $g\in\mathcal A^0(\mathrm{End}(E))$ (after pulling back $(\mathcal V,h)$ to $X_0$ by $\mathrm{id}:X_0\to (X_0,\mu_t)=X_t$)
    \begin{align*}
        h=h_0+th_0g+\bar th_0g^{\star_{h_0}}+O(|t|^2).
    \end{align*}
\end{lemma}
\begin{proof}
    Since $h$ is a Hermitian metric, one has $h=h_0+th_1+\bar t\bar h_1^t+O(|t|^2))$. Then we may assume $h_1=h_0g$ for some $g$ being a section of End$(E)$. Then $\bar h_1^t=\bar g^t h_0=h_0 g^{\star_{h_0}}.$ 
\end{proof}
We also view $h$ as a harmonic map on $\widetilde X$, and by \cite[Lemma 2.16]{Li} we have the following formula for $\psi_h$
\begin{align}\label{psi}
    \psi_h=-\frac{1}{2}h^{-1}Dh.
\end{align}

Combining \eqref{Hmetric} and \eqref{psi}, we will derive a PDE for $g$ later in this section. This $g$ determines how the harmonic metric deforms. Therefore, $g$ determines the first-order deformation class of the isomonodromic deformation of a Higgs bundle.

Let $\omega_t$ be a local (1,0) frame of $X_t$ as in section \ref{Rie}. Let $L_t$ be the dual (1,0) vector field of $\omega_t$ such that $\om_t(L_t)=1$ and $\bar\om_t(L_t)=0$. (For simplicity, we omit $t$ in $\om_t$ and $L_t$ in this section and just denote them by $\om$ and $L$.)

Let $D':=D^{1,0}$ and $D'':=D^{0,1}$ under the complex structure of $X_0$. 
\begin{lemma}\label{Aom}
    By \eqref{psi}, $\theta:=-\cc h_0^{-1}D' h_0$ is the Higgs field of $(\mathcal V,D)$ over $X_0$. Assume $\psi_h=A\om+B\bar\om$ for local sections $A,B$ of $\mathrm{End}E$. Then $\theta_t:=A\om$ is the Higgs field of $(\mathcal V,D)$ over $X_t$. We have the following (after pulling back to $X_0$ by $\mathrm{id}:X_0\to (X_0,\mu_t)=X_t$):
    \begin{align*}
        A\om=\theta+t(-\eta(\theta)+[\theta,g]-\cc D' g)+\bar t(\bar \eta(\theta^{\star_{h_0}})+[\theta,\bg]-\cc D' \bg)+O(|t|^2), \end{align*}
        where $g$ is defined in Lemma \ref{Taylor} and $[\theta,g]=\theta g-g\theta$ .
    And for $B\bar\om$, we have\begin{align*}
        B\bar\om= \theta^{\star_{h_0}}+t(\eta(\theta)+[\theta^{\star_{h_0}},g]-\cc D'' g)+\bar t(-\bar \eta(\theta^{\star_{h_0}})+[\theta^{\star_{h_0}},\bg]-\cc D'' \bg)+O(|t|^2).
    \end{align*}
\end{lemma}
\begin{proof}By \eqref{psi}, we have
\begin{align*}
    \psi_h&=-\frac{1}{2}{(h_0+th_0g+\bar th_0g^{\star_{h_0}})^{-1}}D(h_0+th_0g+\bar th_0\bg)+O(|t|^2)
    \\
    &=-\frac{1}{2} (h_0^{-1} - t gh_0^{-1}-\bar t\bg h_0^{-1}) D(h_0 + t h_0g+\bar th_0\bg) + O(|t|^2)\\
    &=-\cc h_0^{-1}Dh_0+\frac{t}{2}\left\{ gh_0^{-1}Dh_0- h_0^{-1}D(h_0g)\right\}+\frac{\bar t}{2}\left\{\bg h_0^{-1}Dh_0- h_0^{-1}D(h_0\bg)\right\}+O(|t|^2).
\end{align*}

Thus we have 
\begin{equation}\label{AB}
\begin{aligned}
    A=&-\frac{1}{2} h_0^{-1} L h_0 + \frac{t}{2}\left\{ gh_0^{-1}Lh_0- h_0^{-1} L(h_0g)\right\}\\&+\frac{\bar t}{2}\left\{\bg h_0^{-1}Lh_0- h_0^{-1}L(h_0\bg)\right\}+O(|t|^2);\\
    B =& -\frac{1}{2} h_0^{-1} \bar L h_0 + \frac{t}{2}\left\{ gh_0^{-1}\bar Lh_0- h_0^{-1}\bar L(h_0g)\right\}\\&+\frac{\bar t}{2}\left\{\bg h_0^{-1}\bar Lh_0- h_0^{-1}\bar L(h_0\bg)\right\}+O(|t|^2).
\end{aligned}\end{equation}

After some simplification, we obtain the desired result.\end{proof}

Now we determine the PDE that $g$ satisfies. This linear non-homogeneous PDE can be viewed as the linearization of the non-linear equation \eqref{Hmetric} by using the isomonodromic deformation to 'eliminate' the non-linear term. 
\begin{proposition}\label{pde}
    The endomorphism $g$ in Lemma \ref{Aom} satisfies the following PDE
    \begin{align*}
        D''D' g=-2D'(\eta(\theta))-D'([\theta^{\star_{h_0}},g])+D''([\theta,g]);
    \end{align*}
    and $\bg$ satisfies
    \begin{align*}
        D''D'\bg=2D''(\bar \eta(\theta^{\star_{h_0}}))-D'([\theta^{\star_{h_0}},\bg])+D''([\theta,\bg]).
    \end{align*}
    The second equation can be obtained by taking $\star_{h_0}$ of the first equation.
\end{proposition}
\begin{proof}Note that $D_h^{\star_h}\psi_h=\star_h(D_h(\star_h\psi_h))$. Thus we focus on the equation $D_h(\star_h\psi_h)=0.$
Since $\star_h\psi_h=i(A^{\star_h}\bar\om-B^{\star_h}\om)$, we shall compute $A^{\star_h}$ and $B^{\star_h}$. Note that $A^{\star_h}=h^{-1}\overline{A^t}h$ and $B^{\star_h}=h^{-1}\overline {B^t}h$. We have \begin{align*}
    A^{\star_h}=B+O(|t|^2),\ B^{\star_h}=A+O(|t|^2).
\end{align*}
Thus \eqref{Hmetric} is equivalent to $D_h(B\bar\om-A\om)=0$. Now let $a$ be any (local) flat section of $E$. Then we have
\begin{align*}
    &(D_h(B\bar\om-A\om))a
    \\=&D_h\left\{(B\bar\om-A\om)a\right\}+(Ba)d\bar\om-(Aa)d\om+(B\bar\om-A\om)\left\{D_ha\right\}\\
    =&(D-\psi_h)\left\{(B\bar\om-A\om)a\right\}+(Ba)d\bar\om-(Aa)d\om-(B\bar\om-A\om)(\psi_ha)\\
    =&D(Ba)\bar\om+(Ba)d\bar\om-D(Aa)\om-(Aa)d\om\\
    =&D' (Ba)d\bar z-\bar tD'' (Ba)\overline{\eta(dz)}+(Ba)d\bar \om-D'' (Aa)dz\\&+tD' (Aa)\eta(d z)-(Aa)d\om+O(|t|^2)=0.
\end{align*}
Recall \eqref{AB} for $A$ and $B$, one obtains 
\begin{align*}
    &D' (Ba)d\bar z-D''(Aa)dz
    \\=&t\left\{D'(\eta(\theta))+D'([\theta^{\star_{h_0}},g])-D''([\theta,g])+D''D' g\right\}a
    \\&+\bar t\left\{-D''(\bar \eta(\theta^{\star_{h_0}}))+D'([\theta^{\star_{h_0}},\bg])-D''([\theta,\bg])+D''D' \bg\right\}a+O(|t|^2).\end{align*}
Moreover,
     \begin{align*}tD'( Aa)\eta(d z)-A(a)d\om=&tD'(\eta(\theta))a+O(|t|^2);\\
     (Ba)d\bar\om-\bar tD'' (Ba)\overline{\eta(dz)}=&-\bar tD''(\bar \eta(\theta^{\star_{h_0}}))a+O(|t|^2).
\end{align*}
A combination of formulas above gives the required equation. 
\end{proof}

Now we provide another version of Lemma \ref{Aom}, using the Chern connection instead of $D'$ and $D''$. Recall that $$D=D_{h_0}+\theta+\theta^{\star_{h_0}}$$ with $D_{h_0}$ being the Chern connection of $(E,h_0)$ on $X_0$. We have $D_{h_0}^{0,1}=\bar\p$. Hence $$D'=\Ch^{1,0}+\theta \text{ and }D''=\bar\p+\theta^{\star_{h_0}}.$$
 
\begin{proposition}\label{thetadbar}
    Let $\theta_t:=A\om$ be the Higgs field of $(\mathcal V,D)$ over $X_t$ in Lemma \ref{Aom}. We have the following formula (after pulling back to $X_0$)
    \begin{align*}
        \theta_t=\theta+t(-\eta(\theta)+\cc[\theta,g]-\cc \Ch^{1,0} g)+\bar t(\bar \eta(\theta^{\star_{h_0}})+\cc[\theta,\bg]-\cc \Ch^{1,0} \bg)+O(|t|^2). \end{align*}
    The two smooth endomorphisms of $E$, $g$ and $\bg$, satisfy the following equations:
    \begin{equation}\label{equ_g}
    \begin{aligned}
        \bar\p\Ch^{1,0} g&=-2\Ch^{1,0}(\eta(\theta))-[\theta,[\theta^{\star_{h_0}},g]];\\
        \bar\p\Ch^{1,0} \bg&=2\bar\p(\bar \eta(\theta^{\star_{h_0}}))-[\theta,[\theta^{\star_{h_0}},\bg]].
    \end{aligned}\end{equation}
    Besides, the following formula of $\bar\p_t$ (after pulling back to $X_0$) holds
    \begin{align*}\bar\p_t=\bar\p+t(\eta\circ D_{h_0}^{1,0}-\cc [\theta^{\star_{h_0}},g]+\cc \bar\p g)+\bar t(\bar\eta\circ \bar\p-\cc[\theta^{\star_{h_0}},\bg]+\cc\bar\p \bg)+O(|t|^2).
    \end{align*}
\end{proposition}
\begin{proof}
It is straightforward to obtain the formula for $A\om$. 
Let us turn to the PDE of $g$ in Proposition \ref{pde}.
\begin{align*}
    D''D'g&=(\bar\p+\theta^{\star_{h_0}})(\Ch^{1,0}+\theta)g=\bar\p(\Ch^{1,0})g+[\theta^{\star_{h_0}},\Ch^{1,0}g]+[\theta^{\star_{h_0}},[\theta,g]]+\bar\p[\theta,g];\\
    -D'[\theta^{\star_{h_0}},g]&=-\Ch^{1,0}[\theta^{\star_{h_0}},g]-[\theta,[\theta^{\star_{h_0}},g]]=[\theta^{\star_{h_0}},\Ch^{1,0}g]-[\theta,[\theta^{\star_{h_0}},g]];\\
    D''[\theta,g]&=\bar\p[\theta,g]+[\theta^{\star_{h_0}},[\theta,g]];\ -2D'(\eta(\theta))=-2\Ch^{1,0}(\eta(\theta))\text{ as }[\theta,\eta(\theta)]=0.
\end{align*}

Hence we have $\bar\p\Ch^{1,0} g=-2\Ch^{1,0}(\eta(\theta))-[\theta,[\theta^{\star_{h_0}},g]]$. A similar argument applies to $\bg$. The routine computation for the formula of $\bar\p_t$ is omitted here.
\end{proof}

Now we apply the deformation theory established in section \ref{DT}, including Proposition \ref{holtan} and Proposition \ref{aholtan}.
\begin{theorem}\label{main1} (i) The holomorphic
derivative of the isomonodromic section $\sigma$ is given by:
\begin{align*}\Phi^{1,0}_{KS}:T_0S&\to T^{1,0}_{\sigma(0)}\Mdol,\\
\frac{\p}{\p t}&\mapsto [(-\cc\Ch^{1,0}g,-\cc[\theta^{\star_{h_0}},g],\eta)]\end{align*}
where $g$ satisfies the PDE \eqref{equ_g} in Proposition \ref{thetadbar} and $[\eta]=\rho_{KS}(\p/\p t)$;\\
(ii) The anti-holomorphic derivative of the isomonodromic section $\sigma$ is given by:
\begin{align*}\Phi^{0,1}_{KS}:T_0S&\to T_{\sigma(0)}^{0,1}\mathcal M_{\mathrm{Dol}}(X_0)\hookrightarrow T_{\sigma(0)}^{0,1}\Mdol,\\
\frac{\p}{\p t}&\mapsto [(\bar\eta(\theta^{\star_{h_0}}),0)],
\end{align*}
which is the complex conjugation of $\theta_\ast\circ\rho_{KS}(\frac{\p}{\p t})$ by Lemma \ref{thetaast}.
\end{theorem}
\begin{proof}(i) We apply a gauge transform to $(E,\bar\p,\theta)$ to eliminate terms $(\cc[\theta,tg+\bar t\bg],\cc\bar\p(tg+\bar t\bg))$. By Proposition \ref{thetadbar}, we have
\begin{align*}\bar\p_t&=\bar\p+t(\eta\circ D_{h_0}^{1,0}-\cc [\theta^{\star_{h_0}},g])+\bar t(\bar\eta\circ \bar\p-\cc[\theta^{\star_{h_0}},\bg])+O(|t|^2);\\
\theta_t&=\theta-t\eta(\theta)-t(\cc \Ch^{1,0} g)+\bar t(\bar \eta(\theta^{\star_{h_0}})-\cc \Ch^{1,0}\bg)+O(|t|^2).
\end{align*}
By Proposition \ref{holtan}, $[(-\cc\Ch^{1,0}g,-\cc [\theta^{\star_{h_0}},g],\eta)]$ is simply the holomorphic deformation class.\\
(ii) Clearly, by Proposition \ref{aholtan} and the above, $$[(\bar\eta(\theta^{\star_{h_0}})-\cc\Ch^{1,0}\bg,-\cc [\theta^{\star_{h_0}},\bg])]\in T_{\sigma(0)}^{0,1}\mathcal M_{\mathrm{Dol}}(X_0)$$ is the deformation class, which is a harmonic representative. This class is just $[(\bar\eta(\theta^{\star_{h_0}}),0)]\in\overline\Hy$ after modulo the exact term.
\end{proof}

So the class $[(0,\eta(\theta))]\in \Hy$ given by the complex conjugate of $\Phi_{KS}^{0,1}(\p/\p t)$ as above, is the obstruction for the isomonodromic section $\sigma:S\to\Mdol$ being holomorphic along $\frac{\p}{\p t}$. We have the following example where $\Phi_{KS}^{0,1}$ is not zero:

\begin{example}\label{Hiteg} Recall Hitchin's uniformization Higgs bundle $(E:=K_{X_0}^{\cc}\oplus K_{X_0}^{-\cc},\theta=\begin{pmatrix}0&1 \\0 & 0\\\end{pmatrix})$ over $X_0$ defined in \cite[Corollary 4.23]{Hit}. Clearly, $$\theta_\ast: H^1(X_0,T_{X_0})\to \Hy$$
is injective.
\end{example}

We propose the following conjecture:
\begin{conjecture}\label{nfconj}
Let $\Mdol^{0}$ be an irreducible component of $\Mdol$. Then the following statements are equivalent:\\
(a) Any isomonodromic deformation section $\sigma: S\to \Mdol^{0}$ is a holomorphic section;\\
(b) $\Mdol^{0}$ is locally trivial over $S$.
\end{conjecture}

\begin{remark}\label{torelli}In the rank 1 case, i.e., $\mathcal M_{\mathrm{Dol}}(\mathcal X/S,\mathbb C^\ast)=f_\ast\Omega^1_{\mathcal X/S}\times_S \mathrm{Jac}(\mathcal X/S)$, we have:\\
1. By Theorem \ref{main1}, the complex conjugation of $\Phi^{0,1}_{KS}$ is exactly 
\begin{align*}\overline{\Phi^{0,1}_{KS}}:p^\ast TS&\to T^{1,0}(\mathcal M_{\mathrm{Dol}}(\mathcal X/S,\mathbb C^\ast)/S)\\
\{(E,\bar\p,\theta),\frac{\p}{\p t}\}&\mapsto [(0,\eta(\theta))].\end{align*}
Recall $f:\mathcal X\to S$ is a projective family. We have the variation of Hodge structures $R^1f_\ast\mathbb C\otimes \mathcal O_S$ over $S$. Then $\mathcal E:=\mathcal E^{1,0}\oplus \mathcal E^{0,1}=f_\ast\Omega^1_{\mathcal X/S}\oplus R^1f_\ast \mathcal O_\mathcal X$ over $S$ is the Hodge bundle, endowed with the Higgs field $$\varphi:TS\otimes \mathcal E\to TS\otimes \mathcal E^{1,0}\to \mathcal E^{0,1}\hookrightarrow \mathcal E.$$
Over any $s\in S$, $\varphi$ is just the cup-product $\rho_{KS}(\frac{\p}{\p t})\times H^{1,0}(X_s)\to H^{0,1}(X_s)$, where $\frac{\p}{\p t}\in T_sS.$
Note that we have the exponential map $\exp:\mathcal E^{0,1}\to\mathrm{Jac}(\mathcal X/S),$ which induces $$\exp:\mathcal E^{0,1}\times_S \mathcal E^{0,1}\to T^{1,0}(\mathrm{Jac}(\mathcal X/S)/S).$$
Then we have 
\[\begin{tikzcd}TS\otimes(\mathcal E^{1,0}\oplus \mathcal E^{0,1} )\arrow[d,"id\times_S\exp"]\arrow[r,"id_{\mathcal E}\times_S\varphi"] & \mathcal E\times_S \mathcal E\arrow[d,"id_{\mathcal E^{1,0}\times_S \mathcal E^{1,0}}\oplus \exp"]\\
p^\ast TS\arrow[r,"\overline{\Phi^{0,1}_{KS}}"] & T^{1,0}(\mathcal E^{1,0}/S)\oplus T^{1,0}(\mathrm{Jac}(\mathcal X/S)/S)
\end{tikzcd}\]
So we obtain $\varphi$ from $\overline{\Phi^{0,1}_{KS}}$.\\
2. Thus Conjecture \ref{nfconj} holds in the rank 1 case by the above argument and the classical Torelli theorem for compact Riemann surfaces.\\
3. One cannot expect Conjecture \ref{nfconj} to be true for some \textbf{algebraic subvariety} of $\Mdol$. Note that in the rank 1 case, $\mathrm{Jac}(\mathcal X/S)$ is an algebraic subvariety of $\Mdol$ by considering the unitary Higgs bundle. Every isomonodromy deformation of a unitary Higgs bundle is holomorphic since the Higgs field is always zero, and one can use Theorem \ref{main1}. However, $\mathrm{Jac}(\mathcal X/S)\to S$ is not locally trivial.
\end{remark}

\subsection{Isomonodromic deformation of a graded Higgs bundle}
Let $(E,\bar\p,\theta)$ be a \textbf{graded Higgs bundle} on $X_0$, i.e., $(E=\bigoplus\limits_{p=0}^kE^{p,k-p},\theta=\bigoplus\limits_{p=1}^k\theta^{p,k-p})$ such that $\theta^{p,k-p}:E^{p,k-p}\to E^{p-1,k-p+1}\otimes \Omega_{X_0}^1$. Let $\sigma: S\to\Mdol$ be the isomonodromic deformation of $(E,\bar\p,\theta)$ defined in section \ref{Isomd}. Then we have a family of Higgs bundles $(E,\bar\p_t,\theta_t)$ on $X_t$. 

In \cite[Theorem 1.4]{KYZ}, they proved that $(E,\bar\p_t,\theta_t)$ on $X_t$ is a graded Higgs bundle for any sufficiently small $t$ if the class $[(0,\eta(\theta))]\in \Hy$ defined by $\theta_\ast\circ\rho_{KS}(\frac{\p}{\p t})$ vanishes. Equivalently, if the anti-holomorphic Kodaira-Spencer map $\Phi^{0,1}$ defined in Theorem \ref{main1} is non-zero, $(E,\bar\p_t,\theta_t)$ is not a graded Higgs bundle.

We aim to show that for sufficiently small $t$, the Higgs field $\theta_t$ is \textbf{non-nilpotent} along the tangent direction $\frac{\p}{\p t}\in T_0S$ if $\Phi^{0,1}$ is non-zero along this direction.

Firstly, we have a good expression for $g$ in Proposition \ref{thetadbar}:

\begin{lemma}\label{g} The endomorphism $g$ has the decomposition $g=\bigoplus\limits_{p=1}^kg^{p,k-p}$ where $g^{p,k-p}:E^{p,k-p}\to E^{p-1,k-p+1}$ is the composition of $g|_{E^{p,k-p}}$ with $E\to E^{p-1,k-p+1}$.\end{lemma}
\begin{proof}By a direct check and the fact that $h_0=(\bigoplus h_0|_{E^{p,k-p}})$ is block-wise diagonal (proved in \cite[Lemma 3.5]{Li}), the PDE \eqref{equ_g} for $g$ in Proposition \ref{thetadbar} is linear for $g|_{E^{p,k-p}}$ composed with $E\to E^{q,k-q}$ for any $q\ne p-1$. By the uniqueness of the solution of the harmonic metric equation \eqref{Hmetric}, we obtain the desired result.
\end{proof}

\begin{theorem}\label{nonnil} Assume (a) $\sigma(0):=(E,\bar\p,\theta)$ is a graded Higgs bundle on $X$, and \\(b) the non-abelian Kodaira-Spencer map $\Phi^{0,1}_{KS}$ is non-zero at $\sigma(0)$ for some $\frac{\p}{\p t}\in T_0S$.\\
Then the isomonodromic deformation of $(E,\bar\p,\theta)$ along the direction $\frac{\p}{\p t}\in T_0S$ is always non-nilpotent.
\end{theorem}
\begin{proof}Let $[\eta]:=\rho_{KS}(\frac{\p}{\p t})$. We choose a representative $\eta\in \mathcal A^{0,1}(T_{X_0})$ as usual. Then we have $[(0,\eta(\theta)]\ne0\in\Hy$ by Lemma \ref{thetaast}.

Note that $\bar \eta(\theta^{\star_{h_0}})$ consists only of the following smooth sections for $p=0,1,\cdots,k-1$
$$\bar \eta((\theta^{p+1,k-p-1})^{\star_{h_0}}):E^{p,k-p}\to E^{p+1,k-p-1}\otimes \Omega_{X_0}^{1,0},$$
since the harmonic metric $h_0$ is block-wise diagonal.
By Lemma \ref{g}, one has 
\begin{itemize}
\item $\eta(\theta)$ and $\Ch^{1,0}g$ map $E^{p,k-p}$ to $E^{p-1,k-p+1}\otimes \Omega^{1,0}_{X_0}$; 
\item $[\theta,g]$ maps $E^{p,k-p}$ to $E^{p-2,k-p+2}\otimes \Omega^{1,0}_{X_0}$; 
\item $[\theta,\bg]$ maps $E^{p,k-p}$ to $E^{p,k-p}\otimes \Omega^{0,1}_{X_0}$; 
\item $\bar \eta(\theta^{\star_{h_0}})-\cc\Ch^{1,0}\bg$ maps $E^{p,k-p}$ to $E^{p+1,k-p-1}\otimes \Omega^{0,1}_{X_0}$.
\end{itemize}
Then we can compute $\mathrm{tr}(\theta_t^2)$:
\begin{align*}\mathrm{tr}(\theta_t^2)=&\bar t\cdot \mathrm{tr}(\theta(\bar \eta(\theta^{\star_{h_0}})-\cc\Ch^{1,0}\bg)+(\bar \eta(\theta^{\star_{h_0}})-\cc\Ch^{1,0}\bg)\theta)+O(|t|^2)\\
=&\bar t\cdot \mathrm{tr}(\theta(2\bar \eta(\theta^{\star_{h_0}})-\Ch^{1,0}\bg))+O(|t|^2).\end{align*}
It suffices to prove $\mathrm{tr}(\theta(2\bar \eta(\theta^{\star_{h_0}})-\Ch^{1,0}\bg))\not\equiv 0.$
By the PDE \eqref{equ_g} for $\bg$ in Proposition \ref{thetadbar}, we may assume
$$2\bar \eta(\theta^{\star_{h_0}})-\Ch^{1,0}\bg=\phi\in \mathcal A^{1,0}(\mathrm{End}E) \text{ with } \bar\p\phi=[\theta,[\theta^{\star_{h_0}},\bg]].$$
For any $\varphi_1,\varphi_2\in \Ehol$, we have their inner product$$ i\int_{X_0}\mathrm{tr}(\varphi_1\wedge\varphi_2^{\star_{h_0}})\overset{\text{if }\varphi_1=\varphi_2}{=}\|\mathrm{tr}(\varphi_1\wedge\varphi_1^{\star_{h_0}})\|_{L^1}.$$
Thus \begin{align*}&\|4\mathrm{tr}(\bar\eta(\theta^{\star_{h_0}})\wedge(\bar\eta(\theta^{\star_{h_0}}))^{\star_{h_0}})\|_{L^1}\\=&\|\mathrm{tr}(\Ch^{1,0}\bg\wedge(\Ch^{1,0}\bg)^{\star_{h_0}})\|_{L^1}+\|\mathrm{tr}(\phi\wedge\phi^{\star_{h_0}})\|_{L^1}+2\mathrm{Re}(i\int_{X_0}\mathrm{tr} (\phi\wedge (\Ch^{1,0}\bg)^{\star_{h_0}})).
\end{align*}
By the K\"ahler identity for $\Ch^{1,0}$ (see \cite{Simp92} and \cite[Remark 9.2]{Gui}), we have
\begin{align*}&2\mathrm{Re}(i\int_{X_0}\mathrm{tr} (\phi\wedge (\Ch^{1,0}\bg)^{\star_{h_0}}))=2\mathrm{Re}(i\int_{X_0}\mathrm{tr} (i\Lambda_{\omega_0}\bar\p\phi\cdot(\bg)^{\star_{h_0}}))
\\=&2\mathrm{Re}(i\int_{X_0}i\Lambda_{\omega_0}\mathrm{tr} ([\theta,[\theta^{\star_{h_0}},\bg]]\cdot g)
=2\mathrm{Re}(i\int_{X_0}i\Lambda_{\omega_0}\mathrm{tr} ([\theta,g]\wedge[\theta^{\star_{h_0}},\bg])\\=&-2\mathrm{Re}(i\int_{X_0}i\Lambda_{\omega_0}\mathrm{tr} ([\theta,g]\wedge[\theta,g]^{\star_{h_0}})=0.
\end{align*}
Hence we have $\|4\mathrm{tr}(\bar\eta(\theta^{\star_{h_0}})\wedge(\bar\eta(\theta^{\star_{h_0}}))^{\star_{h_0}})\|_{L^1}>\|\mathrm{tr}(\Ch^{1,0}\bg\wedge(\Ch^{1,0}\bg)^{\star_{h_0}})\|_{L^1}$ unless $\phi\equiv 0$. Thus by Cauchy's inequality
\begin{align*}\|4\mathrm{tr}(\bar\eta(\theta^{\star_{h_0}})\wedge(\bar\eta(\theta^{\star_{h_0}}))^{\star_{h_0}})\|_{L^1}>\|2\mathrm{tr}(\Ch^{1,0}\bg\wedge(\bar\eta(\theta^{\star_{h_0}}))^{\star_{h_0}})\|_{L^1}.
\end{align*}
This implies if $\phi\not\equiv0$, then $\|\mathrm{tr}((2\bar\eta(\theta^{\star_{h_0}})-\Ch^{1,0}\bg)\wedge\eta(\theta))\|_{L^1}\ne0.$ Then $\mathrm{tr}(\theta(2\bar \eta(\theta^{\star_{h_0}})-\Ch^{1,0}\bg))\not\equiv 0.$ 

Now we focus on the case where $\phi\equiv 0$, and we have $\bar \eta(\theta^{\star_{h_0}})=\Ch^{1,0}\bg$, i.e., $\eta(\theta)=\bar\p g$. Thus $$0=-i\int_{X_0}\mathrm{tr}(g[\theta,[\theta^{\star_{h_0}},\bg]])=i\int_{X_0}\mathrm{tr}([\theta,g]\wedge[\theta,g]^{\star_{h_0}}),$$ which implies $[\theta,g]\equiv 0$. This contradicts the assumption that $[(0,\eta(\theta))]\in \Hy$ is not zero.
\end{proof}
We propose the following conjecture, as a generalization of Theorem \ref{nonnil}.
\begin{conjecture}\label{nil_nonnil}Assume (a) $\sigma(0):=(E,\bar\p,\theta)$ is a nilpotent Higgs bundle on $X_0$ and \\(b) the non-abelian Kodaira-Spencer map $\Phi^{0,1}_{KS}$ is non-zero at $\sigma(0)$ for some $\frac{\p}{\p t}\in T_0S$.\\
Then the isomonodromic deformation of $(E,\bar\p,\theta)$ is always non-nilpotent along the direction $\frac{\p}{\p t}\in T_0S$.
\end{conjecture}
\begin{remark} By Example \ref{Hiteg}, if $\rho_{KS}$ is injective, Theorem \ref{nonnil} applies to the isomonodromic deformation of Hitchin's uniformization Higgs bundle $(E:=K_{X_0}^{\cc}\oplus K_{X_0}^{-\cc},\theta=\begin{pmatrix}0&1 \\0 & 0\\\end{pmatrix})$ over $X_0$. In this case, its (infinitesimal) isomonodromically deformed Higgs bundles are non-nilpotent. Using \cite[Corollary 4.23]{Hit}, one can prove if $\rho_{KS}$ is injective, then any (not just infinitesimal) isomonodromic deformation of Hitchin's uniformization Higgs bundle is not nilpotent.
\end{remark}
Motivated by the Hitchin's uniformization example, we ask the {\bf distribution} of nilpotent Higgs bundles when doing the isomonodromic deformation:
\begin{conjecture}\label{disconj}Let $\mathcal T_g$ be the Teichm\"uller space. We say the isomonodromic section $\sigma:\mathcal T_g\to \mathcal M_{\mathrm{Dol}}(\mathcal X/\mathcal T_g)$ has nilpotent loci $\mathcal N\subset \mathcal T_g$ if $\mathcal N$ is a union of all maximal irreducible real analytic subvarieties $S_i\subset \mathcal T_g,\ i\in I$, such that $\sigma|_{S_i}$ is always nilpotent. We expect that $\{S_i:i\in I\}$ is a {\bf discrete} set of subvarieties.
\end{conjecture}


\end{document}